\documentclass[11pt]{amsart}
\usepackage{amssymb, latexsym}
\theoremstyle{plain}
\newtheorem{theorem}{Theorem}
\newtheorem{corollary}{Corollary}
\newtheorem*{Z}{Theorem (Zhang)}
\newtheorem*{I}{Theorem (Ishige)}

\newtheorem {lemma}{Lemma}

\theoremstyle{remark}

\newtheorem*{Remark 1}{Remark 1}
\newtheorem*{Remark 2}{Remark 2}

\numberwithin{equation}{section}

\begin{document}

\title[Fujita Exponent for Semilinear Heat Equation]
 {The Fujita Exponent for  Semilinear Heat Equations with  Quadratically Decaying  Potential or in an Exterior Domain }

\author{Ross   Pinsky}
\address{Department of Mathematics\\
Technion---Israel Institute of Technology\\
Haifa, 32000\\ Israel} \email{pinsky@math.technion.ac.il}
\urladdr{http://www.math.technion.ac.il/~pinsky/}

\date{}

\begin{abstract}
Consider the equation
\begin{equation}\label{abstract}
\begin{aligned}
&u_t=\Delta u-Vu +au^p \ \text{in}\ R^n\times (0,T);\\
&u(x,0)=\phi(x)\gneq0,\ \text{in}\ R^n,
\end{aligned}
\end{equation}
where $p>1$, $n\ge2$, $T\in(0,\infty]$,  $V(x)\sim\frac\omega{|x|^2}$ as $|x|\to\infty$, for some $\omega\neq0$, and $a(x)$ is on the order $|x|^m$ as $|x|\to\infty$, for
some $m\in (-\infty,\infty)$.
A solution to the above equation is called global if $T=\infty$. Under some additional technical conditions, we calculate a critical exponent
$p^*$ such that global solutions exist for $p>p^*$, while for $1<p\le p^*$, all solutions blow up in finite time.
We also show that  when $V\equiv0$, the blow-up/global solution dichotomy  for \eqref{abstract} coincides with that for the corresponding
problem in an exterior domain with the Dirichlet boundary condition, including the case in which $p$ is equal to the critical exponent.
\end{abstract}
\maketitle
\section{Introduction and Statement of Results}\label{Intro}

Consider the semilinear heat equation
\begin{equation}\label{prob}
\begin{aligned}
&u_t=\Delta u-Vu +u^p\ \text{in}\ R^n\times(0,T); \\
&u(x,0)=\phi(x)\gneq0 \ \text{in}\ R^n,
\end{aligned}
\end{equation}
where $p>1$, $n\ge1$ and $T\in(0,\infty]$.
 In this paper, when we speak of a \it solution\rm\  to the above equation, or to any of the other equations appearing later on, we mean
\it a classical solution $u$ satisfying $||u(\cdot,t)||_\infty<\infty$, for $0<t<T$.\rm\
 This allows us to employ comparison principles.
A solution to \eqref{prob} is called global if $T=\infty$.
In the case that $V\equiv0$,
$p^*\equiv1+\frac2n$ is the critical exponent, the so-called \it
Fujita exponent\rm, and one has the following dichotomy:
if $p>p^*$, then for sufficiently small
initial data $\phi$, the solution to \eqref{prob} is global,  whereas
if $1<p\le p^*$, then  \eqref{prob} has  no global solution---every solution blows up in
finite time.
This
result goes back to Fujita  \cite{F} in the case $p\neq p^*$.
Various proofs of blow-up in the borderline case $p=p^*$ can be
found in \cite{AW}, \cite{KST}, \cite{P97}.

More recently, Zhang \cite{Z01} considered \eqref{prob} with $n\ge3$ for potentials $V$ behaving like $\frac{\omega}{1+|x|^b}$, for $b>0$ and $\omega\neq0$.
He proved the following result.
\begin{Z} Let $n\ge3$.

\noindent i. If
 $0\le V(x)\le\frac \omega{1+|x|^b}$, for some $b>2$ and $\omega>0$, then $p^*=1+\frac2n$ and consequently the potential does not
 affect the critical exponent;

\noindent ii. If $V(x)\ge\frac \omega{1+|x|^b}$, for some $b\in(0,2)$ and $\omega>0$, then $p^*=1$ and there exist global solutions for all $p>1$;

\noindent iii.  If $ \frac \omega{1+|x|^b}\le V(x)\le 0$, for some $b>2$  and    $\omega<0$ with  $|\omega|$ sufficiently small, then
$p^*=1+\frac2n$ and consequently the potential does not affect the critical exponent;

\noindent iv. If $V(x)\le\frac \omega{1+|x|^b}$, for some  $b\in(0,2)$ and $\omega<0$, then $p^*=\infty$ and there are no global solutions for any $p>1$.
\end{Z}
Note that
wherever the statement of the result is that there exist global solutions,
 Zhang  either does not allow for negative $V$ or else requires that $|V|$ be sufficiently small. The reason for this  will
 become clear from Theorem \ref{spectral} below.

Zhang noted that it seemed difficult to specify the exact value of the critical exponent in the case
of quadratic decay; that is in the case that $V(x)\sim\frac \omega{|x|^2}$ as $|x|\to\infty$. He also noted that it is unclear
whether or not $p^*$ is finite  in the case that $V(x)\sim\frac \omega{|x|^2}$, with $\omega<0$.

Very recently, Ishige \cite{I} treated \eqref{prob}  for $n\ge3$ in
the case $V(x)\sim\frac \omega{|x|^2}$ with $\omega>0$.
Let $\alpha=\alpha(\omega,n)$ denote the larger root of the equation
$\alpha(\alpha+n-2)=\omega$; that is
\begin{equation}\label{root}
\alpha(\omega,n)=\frac{2-n+\sqrt{(n-2)^2+4\omega}}2.
\end{equation}
Since we are assuming here that $\omega>0$, one has $\alpha(\omega,n)>0$.
Define
\begin{equation}\label{pomega}
p^*(\omega)=1+\frac 2{n+\alpha(\omega,n)}.
\end{equation}
\begin{I}
Let $n\ge3$ and assume that  $V\ge0$. Let $\omega>0$.

\noindent i.  If
 $V(x)\ge \frac \omega{|x|^2}$ for large $|x|$, then for $p>p^*(\omega)$
there exist global solutions to \eqref{prob};

\noindent ii. If $V(x)\le\frac\omega{|x|^2}$ for large $|x|$, then for $1<p\le p^*(\omega)$
every  solution to \eqref{prob} blows up in finite time.

\end{I}
Note that Ishige assumes from the outset that $V\ge0$. The delicacy between having global solutions and allowing
$V$ to take negative values will be explained by Theorem \ref{spectral} below.

Ishige's proof involved comparison with a solution to the radially symmetric linear equation
$v_t=\Delta v-\hat V(|x|) v$, where $\hat V(r)\sim\frac\omega{r^2}$ as $r\to\infty$. The large time behavior of this
linear equation, which is needed for the comparison,  was recently obtained by Ishige and Kawakami \cite{IK}.

In this paper, our main focus  is the study of  the remaining case,  $V(x)\sim\frac\omega{|x|^2}$, with $\omega<0$.
 In fact we treat the following more general problem:
\begin{equation}\label{probextra}
\begin{aligned}
&u_t=\Delta u-Vu +au^p \ \text{in}\ R^n\times (0,T);\\
&u(x,0)=\phi(x)\gneq0,\ \text{in}\ R^n,
\end{aligned}
\end{equation}
where $p>1$, $n\ge2$, $T\in(0,\infty]$, $\phi$ is bounded and continuous,
$0\lneq a\in C^\alpha(R^n)$ and
$V\in C^\alpha(R^n-\{0\})$  $\alpha\in(0,1]$.
We also require that $\liminf_{x\to0}V(x)>-\infty$ so that $V$ is locally bounded from below.
Our methods, which are completely different from the method employed by Ishige, also allow one to obtain weaker versions of Ishige's results for the case
$\omega>0$, but in the more general context of equation \eqref{probextra} with $n\ge2$.
The method of proof also leads naturally to a study of the critical exponent in an exterior domain with the Dirichlet boundary
condition in the case $V\equiv0$.

In the case that $V\equiv0$  and that $a$ satisfies
\begin{equation}\label{2sided}
c_1|x|^m\le a(x)\le c_2|x|^m, \ \text{for sufficiently large} \ |x|\ \text{and some}\ m\in (-\infty,\infty), \ c_1,c_2>0,
\end{equation}
 the critical exponent $p^*$ for \eqref{probextra} was calculated  in \cite{P97}; it is given by
\begin{equation}\label{97}
p^*=1+\frac{(2+m)^+}n.
\end{equation}
In \eqref{root} we defined $\alpha(\omega,n)$ for $\omega>0$.
We now extend the definition of $\alpha(\omega,n)$ in \eqref{root} to $\omega\ge-\frac14(n-2)^2$. Note that $\alpha(\omega,n)<0$
for $-\frac14(n-2)^2\le\omega<0$. Now define
\begin{equation}\label{08}
p^*(\omega, m)=1+\frac{(2+m)^+}{n+\alpha(\omega,n)}.
\end{equation}
We will prove the following theorem.
\begin{theorem}\label{negative}
Let $n\ge3$ and let $-\frac14(n-2)^2\le\omega<0$.
Consider \eqref{probextra} with $a(x)$ satisfying \eqref{2sided}.
Assume that $V\in C^\alpha(R^n-\{0\})$ and that \newline $\liminf_{x\to0}V(x)>-\infty$.
Let $p^*(\omega,m)$ be as in \eqref{08}.

\noindent i. If $V(x)\ge\frac\omega{|x|^2}$, then
there exist global solutions to \eqref{probextra} for $p>p^*(\omega,m)$;

\noindent ii. If  $V(x)\le\frac\omega{|x|^2}$, for sufficiently
large $|x|$,
  then there are no global solutions
to \eqref{probextra} for $1<p\le p^*(\omega, m)$.
\end{theorem}
\it\noindent Remark.\rm\
Note that in the case of the existence of global solutions, we allow $V$ to be negative up to a precise globally  specified
size. The reason for this will become clear in Theorem \ref{spectral}.

We now consider what happens when $\omega<-\frac14(n-2)^2$, $n\ge2$.
We will show that $p^*=\infty$ under a certain general condition on the operator $-\Delta+V$, and that this condition holds
if $V(x)\le \frac\omega{|x|^2}$, for $|x|>\epsilon$, with sufficiently small $\epsilon>0$.

Let $D\subseteqq R^n$ be a domain.  Then $-\Delta +V$ on $D$ with the Dirichlet boundary condition
on $\partial D$  can be realized as a self-adjoint operator  on $L^2(D)$. Denoting its spectrum by $\sigma(-\Delta +V;D)$, let
\begin{equation*}
\lambda_{0;D}(-\Delta+V)\equiv\inf\sigma(-\Delta +V;D).
\end{equation*}
\begin{theorem}\label{spectral}
If there exists a domain $D\subseteqq R^n$ for which $\inf_{x\in D}a(x)>0$ and $\lambda_{0;D}(-\Delta+V)<0$, then
 there are no global  solutions to \eqref{probextra} for any $p>1$; that is, $p^*=\infty$.
\end{theorem}
We can use  Theorem \ref{spectral}  to prove the following corollary.
\begin{corollary}\label{corspectral}
Consider \eqref{probextra} with $a>0$ on $R^n$, $n\ge2$. Let
$\omega<-\frac14(n-2)^2$.
There exists an $\epsilon>0$ such that
if   $V(x)\le\frac\omega{|x|^2}$, for $|x|>\epsilon$,  then
there are no global solutions to \eqref{probextra} for any $p>1$; that is, $p^*=\infty$.
\end{corollary}
\begin{Remark 1}
Note that there is a discontinuity in the critical exponent at $\omega=-\frac14(n-2)^2$. By Theorem \ref{negative}, if
$V(x)=-\frac{(n-2)^2}{4|x|^2}$, for sufficiently large $|x|$, and $V(x)\ge-\frac{(n-2)^2}{4|x|^2}$, for all $x$, then
the critical exponent is equal to \linebreak $p^*(-\frac14(n-2)^2,m)=1+\frac{2(2+m)^+}{n+2}$.
However, if $V(x)=\frac\omega{|x|^2}$, for some $\omega<-\frac14(n-2)^2$ and $|x|>\epsilon$, for sufficiently small $\epsilon>0$,
 then the critical exponent
is $\infty$.
\end{Remark 1}
\begin{Remark 2}
Theorem \ref{spectral} makes it clear why in the theorems of  Zhang and of Ishige and in Theorem \ref{negative},
one needed to be careful with regard to stating the existence of  global solutions and allowing $V$ to take negative values.
For example, part (iii) of the theorem of Zhang states that
if $\frac\omega{1+|x|^b}\le V(x)\le 0$ for some $b>2$ and $\omega<0$, with $|\omega|$ sufficiently small, then the critical
exponent for \eqref{prob} is $1+\frac2n$. The requirement that $|\omega|$ be sufficiently small is mandatory in light of
Theorem \ref{spectral}.  Indeed, for any $D\subseteqq R^n$, if $\omega<0$ and $|\omega|$ is sufficiently large, then $\lambda_{0;D}(-\Delta+
\frac\omega{1+|x|^b})<0$
and thus, by Theorem \ref{spectral}, one has $p^*=\infty$.
\end{Remark 2}

The method of proof in Theorem \ref{negative}  also yields the following result for the case $\omega>0$.
\begin{theorem}\label{positive}
Let $n\ge2$ and  $\omega>0$.
Consider \eqref{probextra} with $a(x)$ satisfying \eqref{2sided}.
Assume that $V\in C^\alpha(R^n-\{0\})$.
Let $p^*(\omega,m)$ be as in \eqref{08}.

\noindent i. If $V(x)\ge\frac\omega{|x|^2}$, then there exist global solutions to \eqref{probextra}
for $p>p^*(\omega,m)$.

\noindent ii. If $V(x)\le \frac\omega{|x|^2}$, for sufficiently large $|x|$, then there are no global solutions to \eqref{probextra} for
$1\le p\le p^*(\omega,m)$.
\end{theorem}

\noindent \it Remark.\rm\ Note that part (i) requires that $V$ approach $\infty$ as $|x|\to0$. In fact, as Ishige has proven
in the case that $a\equiv1$ and $n\ge3$, the result should hold as long as $V(x)\ge\frac\omega{|x|^2}$, for sufficiently large $|x|$, and
$V\ge 0$ for all $x$. However, our method of proof  does not seem  to be extendable to this situation.

As will be seen below, the method of  proof we employ for the blow-up case in Theorems \ref{negative} and \ref{positive} will lead naturally
to a consideration of the critical exponent for the semilinear heat equation in an exterior domain with the Dirichlet boundary condition and with $V\equiv0$. Let $B_r=\{x\in R^n: |x|<r\}$.
Consider the following problem:
\begin{equation}\label{exterior}
\begin{aligned}
&u_t=\Delta u +au^p \ \text{in}\ (R^n-\bar B_{r_0})\times (0,T);\\
&u(x,t)=0,\ \text{for}\ |x|=r_0,\ t\ge0;\\
&u(x,0)=\phi(x)\gneq0\ \text{in}\ R^n-\bar B_{r_0},
\end{aligned}
\end{equation}
where
\begin{equation}\label{a-ext}
c_1|x|^m\le a(x)\le c_2|x|^m,\ \text{for sufficiently large }\ |x| \ \text{and some}\ m\in(-\infty,\infty), \ c_1,c_2>0.
\end{equation}
We prove that restricting to an exterior domain does not affect the blow-up/global solution dichotomy.
\begin{theorem}\label{extresult}
Let $n\ge2$. Consider \eqref{exterior} with $a(x)$ satisfying \eqref{a-ext}. Let
$$
p^*=1+\frac{(2+m)^+}n
$$
as in \eqref{97}.

\noindent i. If $1\le p\le p^*$, then there exist global solutions to \eqref{exterior};

\noindent ii. If $p>p^*$, then there are no global solutions to \eqref{exterior}.
\end{theorem}
\noindent \it Remark.\rm\ In the case $a\equiv1$, $n\ge3$ and $p\neq p^*$, the result in Theorem \ref{extresult} was proven in
\cite{BL}. For some other works that treat the critical exponent in exterior domains, see
\cite{LZ} and \cite{Z01a}. Most of the results in these papers do not cover the case in which $p$ is equal to the critical exponent.

We end this section  with an outline of the methods used to prove Theorems \ref{negative} and \ref{positive}, concentrating
on the case of nonexistence of global solutions, which is where our method is novel, and leads to a consideration of the critical
exponent in the case of a semilinear heat equation in an exterior domain with the Dirichlet boundary condition.
By standard comparison techniques, it suffices to treat the radially symmetric case.
 Thus, instead of considering solutions $u(x,t)$ of \eqref{probextra} with $a$ satisfying \eqref{2sided},
we may consider solutions $u(r,t)$ of the equation
\begin{equation}\label{probextrarad}
\begin{aligned}
&u_t=u_{rr}+\frac{n-1}ru_r-V(r)u+a(r)u^p\ \text{in}\ (0,\infty)\times(0,T);\\
&u(r,0)=\phi(r)\gneq0 \ \text{in}\ [0,\infty),
\end{aligned}
\end{equation}
where $p>1$, $T\in(0,\infty]$, $\phi$ is bounded and continuous,
$V\in C^\alpha((0,\infty))$ and $\liminf_{r\to0}V(r)>-\infty$, $0\lneq a\in C^\alpha([0,\infty))$, $\alpha\in(0,1]$, with  $a$ satisfying
\begin{equation}\label{r2sided}
c_1r^m\le a(r)\le c_2 r^m,\ \text{for sufficiently large}\ r \ \text{and some}\ m\in(-\infty,\infty),\ c_1,c_2>0.
\end{equation}

For the existence of global solutions when $p>p^*(\omega,m)$ in part (i) of Theorems \ref{negative} and \ref{positive},
we construct a global super-solution to \eqref{probextrarad}.
Note that in general it is much more difficult to use the method of super/sub-solutions to prove blow-up, since
the construction of an appropriate sub-solution  would probably require a reasonable knowledge of the blow-up profile.

We now turn to the
 nonexistence of global solutions when $1<p\le p^*(\omega,m)$ in part (ii) of Theorems \ref{negative} and \ref{positive},
We may assume without loss of generality that the initial data $\phi$ in \eqref{probextrarad}
satisfy $\phi(r)>0$ for all $r>0$. Indeed, if this is not the case, then for $\delta>0$ sufficiently small, we
can consider $\bar u(r,t)\equiv u(r,t+\delta)$, which also satisfies \eqref{probextrarad} and is strictly positive at $t=0$.
We apply a transformation as follows. Let $u$ be a solution to
\eqref{probextrarad} and define  $v(r,t)=r^{-\alpha}u(r,t)$. Let $\psi(r)=r^{-\alpha}\phi(r)$.
Then one calculates that
\begin{equation}\label{transformed}
\begin{aligned}
&v_t=v_{rr}+\frac{n-1+2\alpha}rv_r+\left(\frac{\alpha(\alpha+n-2)}{r^2}-V(r)\right)v+r^{\alpha(p-1)}a(r)v^p\\
& \text{in}\ (0,\infty)\times(0,T);\\
&v(r,0)=\psi(r)>0\ \text{in}\ (0,\infty).
\end{aligned}
\end{equation}
There will be global solutions of $v$ if and only if there are global solutions of $u$; thus it suffices to study
\eqref{transformed}.
In part (ii) of Theorems \ref{negative} and \ref{positive}, we are assuming  that $V(r)\le\frac\omega{r^2}$, for sufficiently large $r$,
say for $r\ge r_0$, where $\omega\ge-\frac14(n-2)^2$.
If one now chooses $\alpha=\alpha(\omega,n)$ as in \eqref{root}, then the coefficient of $v$  in \eqref{transformed}
is nonnegative for $r\ge r_0$. By the comparison principle, the solution to
 that equation dominates the solution to the equation
\begin{equation}\label{transformedexact}
\begin{aligned}
&w_t=w_{rr}+\frac{N-1}rw_r+\hat a(r)w^p\ \text{in}\ (r_0,\infty)\times(0,T);\\
&w(r,0)=\psi(r)>0\ \text{in}\ [r_0,\infty);\\
&w(r_0,t)=0, \ t>0,
\end{aligned}
\end{equation}
where
\begin{equation}\label{N}
N\equiv n+2\alpha(\omega,n)
\end{equation}
and $\hat a(r)=r^{\alpha(p-1)}a(r)$. In terms of $\hat a$, the assumption \eqref{r2sided} on $a$ is
\begin{equation}\label{hatr2sided}
c_1r^M\le \hat a(r)\le c_2 r^M,\ \text{for sufficiently large}\ r, \ c_1,c_2>0,
\end{equation}
where
\begin{equation}\label{M}
M\equiv\alpha(\omega,n)(p-1)+m, \ m\in(-\infty, \infty).
\end{equation}
(The reason we insisted on $\phi(r)>0$ for all $r>0$, and thus also $\psi(r)>0$ for all $r>0$, is that otherwise we could have ended up with $\psi\equiv0$ in \eqref{transformedexact}.)
Thus, it suffices to show that there are no global solutions to \eqref{transformedexact}-\eqref{M}.

Now \eqref{transformedexact}-\eqref{M}  is the  radial version of  \eqref{probextra}-\eqref{2sided} in the case $V\equiv0$, except that we have placed the Dirichlet boundary condition at $r=r_0$ instead of considering the problem
for all $r>0$, and except that
$m$ in \eqref{2sided} is replaced by $M$ and the dimension
$n$ is replaced by the ``dimension'' $N$. (Note from the definition of $\alpha(\omega,n)$ that
one always has $N\ge2$.)
The critical exponent $p^*$ for \eqref{probextra} with $V\equiv0$ and with $a$ satisfying \eqref{2sided} was given in \eqref{97}.
Substituting $N$ and $M$ for $n$ and $m$ in
\eqref{97}, it is not unreasonable to suspect that no global solutions will    exist if
\begin{equation}\label{inequality}
1<p\le1+\frac{(2+M)^+}N=1+\frac{\left(2+\alpha(\omega,n)(p-1)+m\right)^+}{n+2\alpha(\omega,n)}.
\end{equation}

We now solve \eqref{inequality} for $p$.
Consider first the case $\omega<0$, in which case $\alpha(\omega,n)<0$.
Since we are assuming that $p>1$, \eqref{inequality}  will never hold if
$2+\alpha(\omega,n)(p-1)+m\le0$; that is,
if
\begin{equation}\label{calc1}
p\ge1-\frac{2+m}{\alpha(\omega,n)}.
\end{equation}
On the other hand, if
  $2+\alpha(\omega,n)(p-1)+m>0$, then solving \eqref{inequality} for $p$ gives
\begin{equation}\label{calc2}
1<p\le1+\frac{2+m}{n+\alpha(\omega,n)}.
\end{equation}
One can check that for $m>-2$, the right hand side of \eqref{calc2} is strictly less than the right hand side
of \eqref{calc1}. From this fact along with \eqref{calc1} and \eqref{calc2},  we conclude that
\eqref{inequality} holds if and only
 if $1<p\le p^*(\omega,m)$, where
$p^*(\omega,m)$ is as in \eqref{08}.

Now consider the case $\omega>0$, in which case $\alpha(\omega,n)>0$. \linebreak If
$2+\alpha(\omega,n)(p-1)+m>0$,  then solving \eqref{inequality}  as we did above
gives \eqref{calc2}. On the other hand, if $2+\alpha(\omega,n)(p-1)+m\le0$ (which implies that $m<-2$), then
\eqref{inequality} does not hold. Putting these facts together leads again to
\eqref{inequality} holding if and only if $1<p\le p^*(\omega,m)$.

To turn the above argument into a rigorous proof, we need to show that
indeed no global solutions exist for \eqref{transformedexact}-\eqref{M} when
$1<p\le 1+\frac{(2+M)^+}N$.  That is we need to show that the proof in \cite{P97},
which treated the operator $\Delta$ in $R^n$ (whose radial part is $\frac{d^2}{dr^2}+\frac{n-1}r\frac d{dr}$),
can accommodate two changes: (1)  operators of
the form $\frac{d^2}{dr^2}+\frac{N-1}r\frac d{dr}$ with fractional $N$ and (2)  the Dirichlet boundary condition at $r=r_0$, which serves
to make solutions smaller. The proof in \cite{P97} made rather
heavy use of the explicit form of the heat kernel $p(t,x,y)=(4\pi t)^{-\frac n2}
\exp(-\frac{|y-x|^2}{4t})$ for the corresponding  linear operator $\Delta-\frac \partial{\partial t}$ in $R^n$.
In the present case, the corresponding linear operator is
$\frac{\partial^2}{\partial r^2}+\frac{N-1}r\frac \partial {\partial r}-\frac \partial {\partial t}$
with the Dirichlet boundary condition at $r=r_0$. It turns out that if $N>2$ (equivalently, $\omega>-\frac14(n-2)^2$), then
the heat kernel for this operator  is comparable in an appropriate sense to  the heat kernel for
$\frac{\partial^2}{\partial r^2}+\frac{N-1}r\frac \partial {\partial r}-\frac \partial {\partial t}$ on the entire space $r>0$; thus, we will
be able to use this latter heat kernel, which we can exhibit explicitly.
However, this latter heat kernel
is a much less convenient object than the Gaussian heat kernel. In fact, this obstacle prevented us
from using the method of proof in \cite{P97} to prove the existence of global solutions above the critical exponent; hence the use of super-solutions. However,
we were able to use this heat kernel and  amend the nonexistence proof in \cite{P97} at or below the critical value.
When $N=2$ (equivalently, $\omega=-\frac14(n-2)^2$), the heat kernel with the Dirichlet boundary condition is not comparable to the heat kernel on the whole space,
however an appropriate  lower bound is known and sufficient for our needs.

In section 2 we prove the existence of global solutions in part (i) of Theorems \ref{negative}
and \ref{positive}. In section 3 we prove the nonexistence of global solutions in part (ii) of
Theorems \ref{negative} and \ref{positive}.
In section 4 we prove Theorem \ref{extresult}.
In section 5 we prove Theorem \ref{spectral} and Corollary \ref{corspectral}.

\section{Proofs of Part (i) of Theorems \ref{negative} and \ref{positive}}
We assume that $p>p^*(\omega,m)$, where $p^*(\omega,m)$ is as in \eqref{08}.
As noted in the first section of the paper, instead of studying \eqref{probextra} with $a$ satisfying \eqref{2sided},
it suffices to study the radial problem \eqref{probextrarad} with $a$ satisfying \eqref{r2sided}.
By the standard theory, it suffices to exhibit a global super-solution.
We look for such a  super-solution in the form
\begin{equation*}
v(r,t)=\delta \frac{r^\alpha}{(t+1)^\gamma}\exp(-\frac{cr^2}{t+1}),
\end{equation*}
for some $\delta, c>0$ and some $\alpha,\gamma\in(-\infty,\infty)$.
We have
\begin{equation}\label{vr}
v_r=(\frac\alpha r-\frac{2cr}{t+1})v;
\end{equation}
\begin{equation}\label{vrr}
v_{rr}=(\frac{\alpha^2}{r^2}+\frac{4c^2r^2}{(t+1)^2}-\frac{4c\alpha}{t+1}-\frac\alpha{r^2}-\frac{2c}{t+1})v;
\end{equation}
\begin{equation}\label{vt}
v_t=(-\frac\gamma{t+1}+\frac{cr^2}{(t+1)^2})v.
\end{equation}
The condition on $V$ in part (i) of Theorems \ref{negative} and \ref{positive} is that
 $V(r)\ge\frac\omega{r^2}$, with $-\frac{(n-2)^2}4\le\omega<0$  in Theorem \ref{negative} and $\omega>0$ in
Theorem \ref{positive}.
Using this  along  with \eqref{vr}, \eqref{vrr} and \eqref{vt}, we have
\begin{equation}\label{supersolu}
\begin{aligned}
&v^{-1}(v_{rr}+\frac{n-1}rv_r-V(r)v-v_t+a(r)v^p)\le\\
& (4c^2-c)\frac{r^2}{(t+1)^2}+\frac{\alpha^2+(n-2)\alpha-\omega}{r^2}
+\frac{\gamma-4c\alpha-2cn}{t+1}\\
&+\delta^{p-1}a(r)\frac{r^{\alpha(p-1)}}{(t+1)^{\gamma(p-1)}}\exp(-\frac{c(p-1)r^2}{t+1}).
\end{aligned}
\end{equation}
In order to make the first term on the right hand side of \eqref{supersolu} vanish, we choose
$c=\frac14$, and in order to make the second term on the right hand side of \eqref{supersolu} vanish,
we choose $\alpha=\alpha(\omega,n)$ as in \eqref{root}.

If $m\le0$, the  assumption on $a$ in \eqref{r2sided} guarantees that for some $C>0$, $a(r)\le Cr^m$,  for all $r>0$.
If $m>0$, the assumption on $a$ in \eqref{r2sided} guarantees that for some $C>0$, $a(r)\le C(r\vee1)^m$, for all $r>0$.
This forces us to  break up the next part  of the proof into two cases. We will continue the proof under the assumption
that $m\le0$. After the completion of this case, it will be easy to point out how to handle the case
$m>0$.

Since $a(r)\le Cr^m$,
the final term on the right hand side of \eqref{supersolu} (with $c=\frac14$ and $\alpha=\alpha(\omega,n)$) is  bounded from above by
$C\delta^{p-1}\frac{r^{\alpha(\omega,n)(p-1)+m}}{(t+1)^{\gamma(p-1)}}\exp(-\frac{(p-1)r^2}{4(t+1)})$.
Letting $z=\frac{r^2}{t+1}$, this upper bound can  be written as
$$\frac{C\delta^{p-1}z^{\frac12\alpha(\omega,n)(p-1)+\frac12m}\exp(-\frac14(p-1)z)}{(t+1)^{(\gamma-\frac12\alpha(\omega,n))(p-1)-\frac12m}},
$$
 which is itself  bounded from above by $\frac{C_1C\delta ^{p-1}}{(t+1)^{(\gamma-\frac12\alpha(\omega,n))(p-1)-\frac12m}}$, where \linebreak
$C_1=\sup_{z>0}z^{\frac12\alpha(\omega,n)(p-1)+\frac12 m}\exp(-\frac14(p-1)z)$.
In light of the above analysis, it follows from \eqref{supersolu} that
\begin{equation}\label{v}
v(r,t)=\delta \frac{r^{\alpha(\omega,n)}}{(t+1)^\gamma}\exp(-\frac{r^2}{4(t+1)})
\end{equation}
satisfies
\begin{equation}\label{supersolu2}
\begin{aligned}
&v^{-1}(v_{rr}+\frac{n-1}rv_r-V(r)v-v_t+a(r)v^p)\le\\
&\frac{\gamma-\alpha(\omega,n)-\frac12n}{t+1}+\frac{C_1C\delta^{p-1}}{(t+1)^{(\gamma-\frac12\alpha(\omega,n))(p-1)-\frac12m}}.
\end{aligned}
\end{equation}
If
\begin{equation}\label{gamma1}
\gamma-\alpha(\omega,n)-\frac12n<0
\end{equation}
and
\begin{equation}\label{gamma2}
(\gamma-\frac12\alpha(\omega,n))(p-1)-\frac12m\ge1,
\end{equation}
then after choosing $\delta>0$ sufficiently small, the right hand side of \eqref{supersolu2} will be non-positive.
The two inequalities \eqref{gamma1} and \eqref{gamma2} together are equivalent to
\begin{equation*}
\frac12\alpha(\omega,n)+\frac{1+\frac12 m}{p-1}\le \gamma<\alpha(\omega,n)+\frac12n,
\end{equation*}
and this latter pair of  inequalities can be solved for $\gamma$ if and only if
$\alpha(\omega,n)+\frac{2+m}{p-1}<2\alpha(\omega,n)+n$, or equivalently,
if and only if
$p>1+\frac{2+m}{n+\alpha(\omega,n)}$. Since we have assumed from the outset that $p>1$, we conclude that
if  $p>1+\frac{(2+m)^+}{n+\alpha(\omega,n)}=p^*(\omega,m)$, then
 it is possible
to choose $\gamma$ so that  \eqref{gamma1}
and \eqref{gamma2} hold.

In the case $m>0$, we have $a(r)\le Cr^m$, if $r\ge1$, and $a(r)\le Cr^0$, if $0<r<1$.
Thus, in order for the above analysis to go through in this case, we need to have
\eqref{gamma2} hold as it is written and also with $m$ replaced by 0. However, since $m>0$,
if \eqref{gamma2} holds as it is written, then it holds a fortiori with $m$ replaced by 0.

In the case $\omega>0$,
the function $v$ given by \eqref{v} with $\delta>0$ sufficiently small and $\gamma$ chosen to
satisfy \eqref{gamma1} and \eqref{gamma2} serves as an appropriate global  super-solution.

In the case $\omega<0$, there is one technical problem; namely,
that $\alpha(\omega,n)<0$ and thus $v$ is not finite at $r=0$.
This artificial singularity arises from the use of polar
coordinates. Unfortunately, if one replaces $r$ by $r+c$ for some
$c>0$, then $v$ will no longer be a super-solution. Thus, we argue
as follows. Consider $\omega$ and $p>p^*(\omega,n)$ as fixed. Our
work so far allows us to conclude that for sufficiently small initial data
$\phi$, the solution $u(x,t)$ of \eqref{probextra} satisfies
$u(x,t)\le v(|x|,t)$ up until some possibly finite blow-up time.
Choose  $\epsilon>0$ sufficiently small  so that
$p>p^*(\omega-\epsilon,n)$. The function $v$ in \eqref{v} was
shown to be a super-solution  for \eqref{probextra} under the
assumption that the potential $V$ satisfies $V(x)\ge\frac\omega{|x|^2}$.
Recall that in \eqref{probextra} we are also assuming that $V$ is locally bounded from below.
Therefore, there exists an $r_0>0$ such that $V(x)\ge\frac\omega{r_0^2}$ for $|x|\le r_0$.
One can check that it is then possible to
choose an $x_0\neq0$  such  that
$V(x)\ge\frac{\omega-\epsilon}{|x-x_0|^2}$. Now consider the
radial version \eqref{probextrarad} of \eqref{probextra} but with
the origin shifted to the point $x_0$. Call the new radial variable
$\rho=|x-x_0|$. Since we have $V(\rho)\ge\frac{\omega-\epsilon}{\rho^2}$,
the construction above shows that there exists a function $\hat
v(\rho,t)=\hat\delta \frac{\rho^{\alpha(\omega-\epsilon,n)}}{(t+1)^{\hat
\gamma}}\exp(-\frac{\rho^2}{4(t+1)})$ such that for sufficiently small
initial data $\phi$, the solution $u(x,t)$ of \eqref{probextra}
 satisfies $u(x,t)\le \hat v(|x-x_0|,t)$ up until some possibly finite blow-up
time.  We conclude that for sufficiently small initial data $\phi$, the solution
$u(x,t)$ of \eqref{probextra} satisfies
$u(x,t)\le \hat v(|x-x_0|,t)\wedge v(|x|,t)$ up until  its blow-up time. But the right hand side
is finite for all $x$ and $t$. Thus $u$ is in fact a global solution.

\section{Proofs of Part (ii) of Theorems \ref{negative} and \ref{positive}}
As was shown at the end of section 1, in order to prove that when $1<p\le p^*(\omega,m)$ there are no global solutions to \eqref{probextra}
with $a$ satisfying \eqref{2sided}, it suffices to show that there are no global solutions for \eqref{transformedexact}-\eqref{M}
when $p$ satisfies \eqref{inequality}.
We will always assume that $M>-2$ since otherwise there is nothing to prove.
We wish to employ the method of proof used in \cite{P97}. This method requires a fairly explicit knowledge
of the heat kernel for the corresponding linear equation.
In the present case, the linear equation is $W_t=W_{rr}+\frac{N-1}rW_r$ with $(r,t)\in (r_0,\infty)\times(0,\infty)$, for some possibly fractional $N$ with
$N\ge2$,
and with the Dirichlet boundary condition at $r=r_0$.
Denote the heat kernel for this equation by  $\bar q_{(N, r_0)}(t,r,\rho)$.

Denote  by $q_{(N)}(t,r,\rho)$
  the heat kernel for the equation
$W_t=W_{rr}+\frac{N-1}rW_r$ with $(r,t)\in (0,\infty)\times(0,\infty)$.
The kernel $q_{(N)}(t,r,\rho)$ is the transition probability density for the Bessel process of order $N$, and  is given by \cite{IW}
\begin{equation}\label{qN}
q_{(N)}(t,r,\rho)=\exp(-\frac{r^2+\rho^2}{4t})\frac {\rho^{N-1}}{2t(r\rho)^{\frac N2-1}}I_{\frac N2-1}(\frac{r\rho}{2t}),
\end{equation}
where $I_\nu$ is the modified Bessel function of order $\nu$, given by
\begin{equation}\label{modbess}
I_\nu(x)=(\frac x2)^\nu\sum_{n=0}^\infty\frac{(\frac x2)^{2n}}{n!\Gamma(\nu+n+1)}.
\end{equation}
By the maximum principle, $\bar q_{(N,r_0)}(t,r,\rho)\le q_{(N)}(t,r,\rho)$.
What we need, however, is an appropriate inequality in the reverse direction.

If $N>2$ (equivalently, $\omega>-\frac14(n-2)^2$), then the Bessel process corresponding to the operator $\frac {d^2}{dr^2}+\frac{N-1}r\frac d{dr}$ is transient \cite{P95}. Furthermore,
as will be explained momentarily,
 the uniform
parabolic Harnack inequality holds for the heat equation $W_t=W_{rr}+\frac{N-1}rW_r$ on $r>0$.
Thus,  it follows from \cite{GS}  that there exist constants $K_0,c>0$ such that
\begin{equation}\label{qqbar}
\bar q_{(N, r_0)}(t,r,\rho)\ge  cq_{(N)}(K_0t,r,\rho), \ \text{for}\ r>r_0+1, \ \rho>r_0+1, \  t>0\ \text{and}\ N>2.
\end{equation}
(The uniform parabolic Harnack inequality concerns nonnegative solutions $W$ of $W_t=W_{rr}+\frac{N-1}rW_r$ on $r>0$ on a time interval
$[\tau,\tau+T]$ . See \cite[Definition 2.2]{GS}
for the precise definition. Any such solution can be represented as
$W(r,\tau+t)=\int_0^\infty q_{(N)}(t,r,\rho)W(\rho,\tau)d\rho$, $0\le t\le T$.
Using the explicit
formula for $q_{(N)}$ in \eqref{qN}, one can verify the uniform parabolic Harnack inequality. Indeed, in the case that $N$ is an integer,
the above heat equation is just the radial form of the standard heat equation on $R^N$, and it is well-known that
the uniform Harnack inequality holds in this case \cite{LY}.)

The following key a priori lower bound on  solutions to \eqref{transformedexact}-\eqref{M}
in the case that $1<p\le 1+\frac{(2+M)^+}N$ will be used to prove the theorem. Then we will come back to prove the lemma.
\begin{lemma}\label{key}
Let $w$ be a  solution to \eqref{transformedexact}-\eqref{M}  on a time interval $0<t<T$,  with $1<p\le 1+\frac{(2+M)^+}N$ and $N>2$.
Then for some $K,C>0$,
\begin{equation}\label{oldresult2}
 w(r,t)\ge Ct^{-\frac N2}\log(1+t)\exp(-\frac{Kr^2}t),\ \text{for}\ 2<t<T,\ r>r_0+1.
\end{equation}
\end{lemma}
\noindent \it Remark.\rm\ The proof of Lemma \ref{key} makes use of \eqref{qqbar}. If $N=2$, a weaker lower bound holds
for $\bar q_{(N, r_0)}$ in terms of $q_{(N)}$. This weaker bound is enough to prove \eqref{oldresult2} when $N=2$
with the restriction that $r\ge t^\frac12$.
 See Lemma \ref{n=2}
and \eqref{GS2}. As the proof of Theorem \ref{negative} below shows, it is enough to have the estimate \eqref{oldresult2} for
$r\ge t^\frac12$.

In light of the above remark, \eqref{oldresult2} holds  for all $N\ge2$ and $r\ge t^\frac12$. We now use this to prove the theorem.
\begin{proof}[Proof of Theorem \ref{negative}]
Assume that $w(r,t)$ is a global solution to \eqref{transformedexact}-\eqref{M}.
For $n>r_0+1$, define
\begin{equation*}
F_n(t)=\int_n^{2n}w(r,t)\phi^{(n)}(r)r^{N-1}dr,
\end{equation*}
where $\phi^{(n)}>0$, normalized by $\int_n^{2n}\phi^{(n)}(r)r^{N-1}dr=1$, is the eigenfunction corresponding to the principal eigenvalue $\lambda_n>0$
for the operator  $-(\frac{d^2}{dr^2}+\frac{N-1}r\frac d{dr})=-r^{1-N}\frac d{dr} r^{N-1}\frac d{dr}$ on $(n,2n)$ with the Dirichlet boundary condition at the endpoints.
For an appropriate value of $n$, we will show that $F_n$ blows up in finite time, thereby contradicting the assumption that $w$ is
a global solution.

From the outset, we assume that $n$ is sufficiently large so that \eqref{hatr2sided} holds for $r\ge n$.
Simple scaling shows that $\lambda_n$ is on the order $\frac1{n^2}$ as $n\to\infty$. In particular then, there exists a constant $c>0$ such
that
$\lambda_n\le \frac c{n^2}$.
Since $\phi^{(n)}(n)=\phi^{(n)}(2n)=0$, one has  $(\phi^{(n)})'(n)\ge0$ and $(\phi^{(n)})'(2n)\le0$.
Using the facts in this paragraph, integrating by parts and using Jensen's inequality, we have
\begin{equation}\label{intbyparts}
\begin{aligned}
&F_n'(t)=\int_n^{2n}w_t(r,t)\phi^{(n)}(r)r^{N-1}dr\\
&=\int_n^{2n}\left(w_{rr}(r,t)+\frac{N-1}rw_r(r,t)+\hat a(r)w^p(r,t)\right)\phi^{(n)}(r)r^{N-1}dr\\
&=\int_n^{2n}(r^{N-1}w_r(r,t))_r\phi^{(n)}(r)dr+\int_n^{2n}\hat a(r)w^p(r,t)\phi^{(n)}(r)r^{N-1}dr\\
&\ge\int_n^{2n}(r^{N-1}\phi^{(n)}_r(r))_rw(r,t)dr+c_1n^M\int_n^{2n}w^p(r,t)\phi^{(n)}(r)r^{N-1}dr\\
&=-\lambda_n F_n(t)+c_1n^M\int_n^{2n}w^p(r,t)\phi^{(n)}(r)r^{N-1}dr\\
&\ge-\frac c{n^2}F_n(t)+c_1n^MF_n^p(t).
\end{aligned}
\end{equation}

The function $-\frac c{n^2} x+c_1n^Mx^p$ is both positive and increasing for $x>(\frac c{c_1})^{\frac1{p-1}}n^{-\frac{M+2}{p-1}}$. Therefore,
if there exists an $n$ and
a $T_n$ for which $F_n(T_n)>(\frac c{c_1})^{\frac1{p-1}}n^{-\frac{M+2}{p-1}}$, then it follows from \eqref{intbyparts}
and the fact that $p>1$ that $F_n(t)$ will blow up at some finite
value of $t$. From Lemma \ref{key} and the remark following it, we obtain $w(r,n^2)\ge C_1n^{-N}\log n$, for $n\le r\le 2n$ and some $C_1>0$.
Thus, $F_n(n^2)\ge C_1n^{-N}\log n$. Since $1<p\le 1+\frac{(2+M)^+}N$, one can choose $n$ sufficiently large so that
$F_n(n^2)\ge C_1n^{-N}\log n>(\frac c{c_1})^{\frac1{p-1}}n^{-\frac{M+2}{p-1}}$.
\end{proof}
\begin{proof}[Proof of Lemma \ref{key}]
The solution $W$ to the corresponding linear problem $W_t=W_{rr}+\frac{N-1}rW_r$  with the Dirichlet boundary condition at $r=r_0$ and
with initial data $\psi$ is given by
\begin{equation}\label{linear}
W(r,t)=\int_{r_0}^\infty \bar q_{(N, r_0)}(t,r,\rho)\psi(\rho)d\rho.
\end{equation}
By comparison, the solution $w$ to \eqref{transformedexact}-\eqref{M} satisfies
\begin{equation}\label{est-w}
w\ge W.
\end{equation}
On the other hand, the solution $w$ to \eqref{transformedexact}-\eqref{M} satisfies the inequality
\begin{equation}\label{intequ}
w(r,t)\ge\int_{r_0}^\infty \bar q_{(N,r_0)}(t,r,\rho)\psi(\rho)d\rho+\int_0^tds\int_{r_0}^\infty d\rho\
\bar q_{(N,r_0)}(t-s,r,\rho)\hat a(\rho)w^p(\rho,s).
\end{equation}
(See \cite{P97} and \cite{Z99}, where it is also shown that under appropriate conditions, \eqref{intequ} holds with an equality.)
Without loss of generality, we assume that $r_0+2$ is contained in the support of $\psi$ appearing in \eqref{linear}.
From
\eqref{qN}-\eqref{qqbar} and \eqref{linear}-\eqref{intequ} it then follows  that
\begin{equation}\label{keyest}
w(r,t)\ge c_1\int_0^tds\int_{r_0+1}^\infty d\rho\ q_{(N)}(K_1(t-s),r,\rho)\hat a(\rho)q^p_{(N)}(K_1s,\rho,r_0+2),\ r>r_0+1,
\end{equation}
for some $K_1,c_1>0$.

In the case that $N$ is an integer, which we denote by $N_0$, $q_{(N_0)}$ is just the standard $N_0$-dimensional Gaussian heat kernel in radial
coordinates, and \eqref{keyest} can be converted back to $N_0$-dimensional Euclidean coordinates. In \cite{P97}, the right
hand side of \eqref{keyest} (converted to Euclidean coordinates and with some other inessential differences) was shown to satisfy the inequality
\begin{equation}\label{oldresult1}
\begin{aligned}
&\int_1^{\frac t2}ds\int_{r_0+1}^\infty d\rho\ q_{(N_0)}(K_1(t-s),r,\rho)\hat a(\rho)q^p_{(N_0)}(K_1s,\rho,r_0+2)\ge\\
&\begin{cases}&
Ct^{1-\frac {N_0}2p+\frac M2}\exp(-\frac{Kr^2}t),\ \text{if}\ p<1+\frac{2+M}{N_0},\\
& Ct^{-\frac {N_0}2}\log(1+t)\exp(-\frac{Kr^2}t),\ \text{if}\ p=1+\frac{2+M}{N_0},
\end{cases}\\
& \text{for}\ t>2,\ r>r_0+1,
\end{aligned}
\end{equation}
where $K,C>0$.
Recall that we are assuming that $M>-2$.
Note that $1-\frac {N_0}2 p+\frac M2>-\frac {N_0}2$, if $p<1+\frac{2+M}{N_0}$, and
$1-\frac {N_0}2p+\frac M2=-\frac {N_0}2$, if $p=1+\frac{2+M}{N_0}$.
Thus,
from \eqref{keyest} and \eqref{oldresult1} it follows immediately that
\eqref{oldresult2} holds for $N=N_0$.
 (For \eqref{oldresult1} and \eqref{oldresult2} with $N=N_0$, see  the statements and proofs of  \cite[Lemma 2, Proposition 1 and Lemma 3]{P97}.
 The spatial integral in \cite{P97} is over all of $R^{N_0}$, which would correspond here to $\rho>0$. But one could have worked just as well
 with $|x|>r_0+1$ in \cite{P97}, so the restriction here to $\rho>r_0+1$ in the spatial integral causes no problem.)

We now proceed to demonstrate that  \eqref{oldresult1}, and consequently also \eqref{oldresult2}, continue to hold
 in the case that $N_0$ is replaced by any non-integral $N>2$. We write $N=N_0-\beta$, where
$N_0\ge3$ is an integer and $\beta\in(0,1)$.
Let $K_\nu(x)\equiv(\frac x2)^{-\nu}I_\nu(x)$, and note from the definition of $I_\nu$  in \eqref{modbess} that $K_\nu(x)$ is decreasing
in $\nu$. Thus, we have from \eqref{qN},
\begin{equation}\label{NN0}
\begin{aligned}
&q_{(N)}(t,r,\rho)=\exp(-\frac{r^2+\rho^2}{4t})\frac{\rho^{N-1}}{2t(r\rho)^{\frac N2-1}}(\frac{r\rho}{2t})^{\frac N2-1}K_{\frac N2-1}(\frac{r\rho}{2t})\\
&=\exp(-\frac{r^2+\rho^2}{4t})\frac{\rho^{N_0-1}}{2t(r\rho)^{\frac {N_0}2-1}}(\frac{r\rho}{2t})^{\frac {N_0}2-1}
K_{\frac N2-1}(\frac{r\rho}{2t})\left(\frac{\rho^{-\beta}}{(r\rho)^{-\frac\beta2}}(\frac{r\rho}{2t})^{-\frac\beta2}\right)\\
&\ge\frac{(2t)^{\frac\beta2}}{\rho^\beta}q_{(N_0)}(t,r,\rho).
\end{aligned}
\end{equation}
From \eqref{NN0} we have
\begin{equation}\label{final}
\begin{aligned}
&q_{(N)}(K_1(t-s),r,\rho)q^p_{(N)}(K_1s,\rho,r_0+2)\\
&\ge C_1\frac{t^\frac\beta2}{\rho^\beta}s^{\frac\beta2p}q_{(N_0)}(K_1(t-s),r,\rho)q^p_{(N_0)}(K_1s,\rho,r_0+2), \\
&\text{for}\ 1\le s\le \frac t2,\ 0\le\rho<\infty,
\end{aligned}
\end{equation}
for some $C_1>0$.
From \eqref{final} we have
\begin{equation}\label{connectNN0}
\begin{aligned}
&\int_1^{\frac t2}ds\int_{r_0+1}^\infty d\rho\ q_{(N)}(K_1(t-s),r,\rho)\hat a(\rho)q^p_{(N)}(K_1s,\rho,r_0+2)\ge\\
&C_1t^{\frac\beta2}\int_1^{\frac t2}ds\int_{r_0+1}^\infty d\rho\ \rho^{-\beta}s^{\frac\beta2p}q_{(N_0)}(K_1(t-s),r,\rho)\hat a(\rho)q^p_{(N_0)}(K_1s,\rho,r_0+2).
\end{aligned}
\end{equation}

Note that the only difference between the terms appearing inside the  double integral on the right hand side of \eqref{connectNN0} and the
terms appearing inside the double integral on the left hand side of \eqref{oldresult1}
is the addition of the factors $\rho^{-\beta}$ and $s^\frac\beta2$.
Translating the setup and notation in the proof of \eqref{oldresult1} in \cite{P97} to the present situation, we note that the integration over $\rho$
introduced a term of the form $((t-s)r(s,t))^{\frac M2}$, where $r(s,t)=\frac s{s+pK_2(t-s)}$, for some $K_2>0$, and the exponent $\frac M2$
was a consequence of  $\hat a$ being on the order $\rho^M$. Since $\hat a(\rho)$ is replaced
by $\rho^{-\beta}\hat a(\rho)$ in \eqref{connectNN0},  in the present situation we obtain a term of the form
$((t-s)r(s,t))^{\frac M2-\frac\beta2}$; see \cite[(2.34)-(2.37)]{P97}.
Thus, whereas in  the penultimate step in the proof of \eqref{oldresult1} in \cite{P97} we obtained
\begin{equation*}
\begin{aligned}
&\int_1^{\frac t2}ds\int_{r_0+1}^\infty d\rho\ q_{(N_0)}(K_1(t-s),r,\rho)\hat a(\rho)q^p_{(N_0)}(K_1s,\rho,r_0+2)\ge\\
&C_2\exp(-\frac {Kr^2}t)\int_1^{\frac t2}s^{-\frac {N_0}2p}(r(s,t))^{\frac {N_0}2+\frac M2}(t-s)^\frac M2ds,
\end{aligned}
\end{equation*}
for some $K>0$ (see \cite[(2.37)]{P97}),
we obtain here
\begin{equation}\label{new237}
\begin{aligned}
&t^{\frac\beta2}\int_1^{\frac t2}ds\int_{r_0+1}^\infty d\rho\ \rho^{-\beta}s^{\frac\beta2p}q_{(N_0)}(K_1(t-s),r,\rho)\hat a(\rho)q^p_{(N_0)}(K_1s,\rho,r_0+2)\\
&\ge C_2t^{\frac\beta2}\exp(-\frac {Kr^2}t)\int_1^{\frac t2}s^{-\frac {N_0}2p+\frac\beta2p}(r(s,t))^{\frac {N_0}2+\frac M2-\frac\beta2}(t-s)^{\frac M2-\frac\beta2}ds.
\end{aligned}
\end{equation}
Making the change of variables $u=\frac st$ and recalling that $N_0-\beta=N$, we have
\begin{equation}\label{last}
\begin{aligned}
&t^{\frac\beta2}\int_1^{\frac t2}s^{-\frac {N_0}2p+\frac\beta2p}(r(s,t))^{\frac {N_0}2+\frac M2-\frac\beta2}(t-s)^{\frac M2-\frac\beta2}ds=\\
&t^{1+\frac M2-\frac N2p}\int_{\frac1t}^{\frac12}u^{\frac N2+\frac M2-\frac N2p}(u+pK_2(1-u))^{-\frac N2-\frac M2}(1-u)^{\frac M2-\frac\beta2}du.
\end{aligned}
\end{equation}
If $p<1+\frac{2+M}N$, then $\frac N2+\frac M2-\frac N2p>-1$ and the integral on the right hand side of \eqref{last} is bounded in $t$. However
if $p=1+\frac{2+M}N$, then $\frac N2+\frac M2-\frac N2p=-1$ and that integral is on the order of $\log t$.
Using this fact along with \eqref{connectNN0}-\eqref{last}, we conclude that \eqref{oldresult1} holds with the integer $N_0$ replaced by  non-integral  $N$.
From this and \eqref{keyest} we then also obtain \eqref{oldresult2} with the integer $N_0$ replaced by non-integral $N$.
This completes the proof of Lemma \ref{key}.

\end{proof}

\section{Proof of Theorem \ref{extresult}}
Note
that \eqref{transformedexact}-\eqref{M} with $N$ equal to an integer is the radial version of \eqref{exterior}-\eqref{a-ext} (with $N$ and $M$ identified
with $n$ and $m$). Thus, in fact,
Lemma \ref{key} and the proof of Theorem \ref{negative} given in section 3 give a proof of Theorem \ref{extresult} in the case $n\ge3$. If we prove
the equivalent of Lemma \ref{key} for $n=2$, then we will also have a proof of Theorem \ref{extresult} for $n=2$.
In fact, as the proof of Theorem \ref{negative} showed, it suffices to have the estimate on $w(r,t)$ in Lemma \ref{key}
for $r\ge t^\frac12$.
 Thus, it suffices to prove the following result.
\begin{lemma}\label{n=2}
Let $w$ be a  solution to \eqref{exterior} with $n=2$ on a time interval $0<t<T$,  with $1<p\le 1+\frac{(2+m)^+}2$.
Then for some $K,C>0$,
\begin{equation}\label{oldresult22}
 w(x,t)\ge Ct^{-1}\log(1+t)\exp(-\frac{K|x|^2}t),\ \text{for}\ |x|>t^\frac12\ \text{and}\ 5<t<T.
\end{equation}
\end{lemma}
\begin{proof}
We assume that $m>-2$ since otherwise there is nothing to prove.
Let $p(t,x,y)=(4\pi t)^{-1}\exp(-\frac{|y-x|^2}{4t})$ denote the heat kernel for the Laplacian on $R^2$, and
let $\bar p_{r_0}(t,x,y)$ denote the corresponding heat kernel for the Laplacian on $R^2-\bar B_{r_0}$ with the Dirichlet boundary condition
at $|x|=r_0$.
It was shown in \cite{GS} that for appropriate constants $c_0,K_0>0$,
one has
\begin{equation}\label{GS2}
\begin{aligned}
&\bar p_{r_0}(t,x,y)\ge c_0\frac{\log(1+|x|)\log(1+|y|)}{\left(\log(1+\sqrt t)+\log(1+ |x|)\right)\left(\log(1+\sqrt t)+\log (1+|y|\right)}p(K_0t,x,y),\\
& \text{for}\ |x|>r_0+1,\ |y|> r_0+1,\ t>0.
\end{aligned}
\end{equation}

We now follow   to a significant degree the proof of blow-up in \cite{P97}. Similar to  \cite[Lemma 1]{P97}, we have
\begin{equation}\label{integralequ}
w(x,t)\ge\int_{R^2-\bar B_{r_0}}\bar p_{r_0}(t,x,y)\phi(y)dy+\int_0^t\int_{R^2-\bar B_{r_0}}\bar p_{r_0}(t-s,x,y)a(y)w^p(y,s)dyds.
\end{equation}
The first term on the right hand side of \eqref{integralequ},
which is the solution of the corresponding linear problem,
constitutes a lower bound for $w$. Thus, using \eqref{GS2}, we have similar to
\cite[Lemma 2]{P97},
\begin{equation}\label{linearbound}
w(x,t)\ge ct^{-1}\exp(-\frac{|x|^2}{2K_0t})
\frac{\log(1+|x|)}{\left(\log(1+\sqrt
t)+\log(1+|x|)\right)\left(\log(1+\sqrt t)\right)},
\end{equation}
for some $c>0$.
Note that for $|x|\ge t^\frac12$, $|y|\ge t^\frac14$ and $t\ge1$, the expression
$\frac{\log(1+|x|)\log(1+|y|)}{\left(\log(1+\sqrt t)+\log(1+ |x|)\right)\left(\log(1+\sqrt t)+\log (1+|y|\right)}$
is bounded and bounded away from 0. Thus,
substituting the estimate \eqref{linearbound} into the second term on the right hand side of \eqref{integralequ}, and using \eqref{GS2} and \eqref{a-ext},
it follows  that
for some $C>0$,
\begin{equation}\label{goodbound}
\begin{aligned}
&w(x,t)\ge \frac Ct\int_{t^\frac12}^{\frac12t}\int_{|y|>t^\frac14}s^{-p}|y|^m\exp(-\frac{|y-x|^2}{Ct})\exp(-\frac{|y|^2p}{2K_0s})dyds,\\
&\text{for}\ |x|\ge t^\frac12\ \text{and large}\ t.
\end{aligned}
\end{equation}
Performing some algebraic manipulations similar to those in \cite[p.166]{P97},
one has for $t,s\ge1$ and   some $c>0$,
\begin{equation}\label{manipulations}
\exp(-\frac{|y-x|^2}{Ct})\exp(-\frac{|y|^2p}{2K_0s})\ge\exp(-\frac{|x|^2}{ct})\exp(-\frac{|y|^2}{cs}).
\end{equation}
Recalling that $m>-2$ and that $n=2$, it is not hard to show, similar to \cite[Lemma 4]{P97}, that for some $k>0$,
\begin{equation}\label{m-int}
\int_{|y|>t^\frac14}|y|^m\exp(-\frac{|y|^2}{cs})dy\ge ks^{1+\frac m2}, \ \text{for}\  s\ge t^\frac12.
\end{equation}
From \eqref{goodbound}-\eqref{m-int}, we obtain for some $k_1>0$,
\begin{equation}\label{final}
w(x,t)\ge\frac{k_1}t\exp(-\frac{|x|^2}{ct})\int_{t^\frac12}^{\frac12t}s^{1+\frac m2-p}ds,\ \text{for}\ |x|\ge t^\frac12\ \text{and large}\ t.
\end{equation}
By assumption, $1<p\le1+\frac{2+m}2=2+\frac m2$; thus, $1+\frac m2-p\ge-1$. Consequently, for some $k_2>0$ and $t\ge5$, we have
\begin{equation}\label{int-s}
\int_{t^\frac12}^{\frac12t}s^{1+\frac m2-p}ds\ge k_2\log t.
\end{equation}
Now \eqref{oldresult22} follows from \eqref{final} and \eqref{int-s}.
\end{proof}

\section{Proofs of Theorem \ref{spectral} and Corollary \ref{corspectral}}
\noindent \it Proof of Theorem \ref{spectral}.\rm\
It is known that $\lambda_{0;D}(-\Delta +V)$ is non-increasing in $D$ and that $\lambda_{0;D}(-\Delta +V)=\lim_{k\to\infty}
\lambda_{0;D_k}(-\Delta +V)$, if $D_k\uparrow D$ \cite[chapter 4]{P95}.
These properties of $\lambda_{0;D}(-\Delta+V)$  allow us to assume without loss of generality
 that the domain $D$ in the statement of the  theorem is    bounded and has a smooth boundary.
 As such,  $\lambda_{0;D}(-\Delta+V)<0$ is in fact the principal eigenvalue for $-\Delta+V$ in $D$ with the Dirichlet
boundary condition. Let $\psi_0>0$,
normalized by $\int_D\psi_0(x)dx=1$, denote the corresponding eigenfunction.

Assume now that $u(r,t)$ is a global solution to \eqref{probextra} for some $p>1$.
  Define
\begin{equation}
F(t)=\int_Du(x,t)\psi_0(x)dx.
\end{equation}
We will show that $F$ blows up at some finite time, thereby contradicting the assumption that $u$ is a global solution.
Note that $\psi_0$ vanishes on $\partial D$ and that $\nabla\psi_0\cdot \nu\le0$ on $\partial D$, where
$\nu$ is the unit outward normal to $D$ at $\partial D$. Also, by assumption $\inf_{x\in D}a(x)\ge\delta$, for some $\delta>0$.
Integrating by parts, and using Jensen's inequality and the facts above, we have
\begin{equation}\label{Fprime}
\begin{aligned}
&F'(t)=\int_Du_t(x,t)\psi_0(x)dx=\int_D(\Delta u-Vv+au^p)(x)\psi_0(x)dx\\
&\ge-\lambda_{0;D}(-\Delta+V)F(t)+\delta F^p(t)\ge\delta F^p(t).
\end{aligned}
\end{equation}
Although the initial data $\phi$ of $u$ may vanish identically on $D$, one certainly has $F(t)>0$ for $t>0$. Thus, it
follows from \eqref{Fprime} and the fact that $p>1$ that $F$ blows up at some finite time.
\hfill $\square$
\medskip

\noindent \it Proof of Corollary \ref{corspectral}.\rm\
It is well-known that $\lambda_{0;R^n-\{0\}}(-\Delta+\frac \gamma{|x|^2})<0$ if $\gamma>\frac{(n-2)^2}4$ \cite[pp. 153-154]{P95}.
Let  $B_k=\{x\in R^n:|x|<k\}$.
Recalling the facts noted in the first line of the proof of Theorem \ref{spectral},
it follows that $\lambda_{0;B_k-\bar B_\epsilon}(-\Delta+\frac\gamma{|x|^2})<0$, for sufficiently large $k$ and sufficiently small $\epsilon>0$.
Since $a$ is continuous and positive by assumption, it follows that $a$ is bounded away from 0 on $B_k-\bar B_\epsilon
$. Thus, the corollary follows from
Theorem \ref{spectral}.
\hfill $\square$

\end{document}